\documentclass[11pt]{amsart}
\usepackage{amssymb,latexsym,amsmath,amsfonts}
\usepackage[mathscr]{eucal}
\usepackage{mathrsfs}
\usepackage{amscd}
\usepackage{amsthm}
\usepackage{amsxtra}
\usepackage[nointegrals]{wasysym}
\usepackage{bm}
\usepackage{calc}
\usepackage{float}
\usepackage[usenames,dvipsnames]{xcolor}
\usepackage{hyperref}
\usepackage{hyperref}
\hypersetup{colorlinks=true, citecolor=MidnightBlue, linkcolor=BrickRed}
\usepackage{mathtools}
\DeclarePairedDelimiter\floor{\lfloor}{\rfloor}

\numberwithin{equation}{section}
\theoremstyle{definition}

\theoremstyle{definition}

\newtheorem*{ques}{Question}

\newtheorem{remark}{Remark}[section]
\theoremstyle{plain}
\newtheorem{theorem}{Theorem}[section]
\newtheorem{lemma}[theorem]{Lemma}
\newtheorem{cor}[theorem]{Corollary}
\newtheorem{Prop}[theorem]{Proposition}

\newcommand{\beas}{\begin{eqnarray*}}
\newcommand{\eeas}{\end{eqnarray*}}
\newcommand{\bes} {\begin{equation*}}
\newcommand{\ees} {\end{equation*}}
\newcommand{\be} {\begin{equation}}
\newcommand{\ee} {\end{equation}}
\newcommand{\bea} {\begin{eqnarray}}
\newcommand{\eea} {\end{eqnarray}}

\newcommand{\eps}{\varepsilon}

\newcommand{\om}{\omega}

\newcommand\partl[2]{\dfrac{\partial{#1}}{\partial{#2}}}

\newcommand{\cont}{\mathcal{C}}
\newcommand{\hol}{\mathcal{O}}

\newcommand{\op}{\operatorname}
\newcommand{\wt}{\widetilde}

\newcommand{\CC}{\mathbb{C}^2}
\newcommand{\Cn}{\mathbb{C}^n}
\newcommand{\C} {\mathbb{C}} 

\newcommand{\rl}{\mathbb{R}}

\newcommand{\pnat} {\mathbb{N}_+} 
\newcommand{\N} {\mathbb{N}}
\newcommand{\RP}{\mathbb{RP}}
\newcommand{\CP}{\mathbb{CP}}

\newcommand{\std}{_{\operatorname{st}}}

\begin{document}
\title{Rational and Polynomial Density on Compact Real Manifolds}
\author{Purvi Gupta and Rasul Shafikov}
\address{Department of Mathematics, University of Western Ontario, London, Ontario N6A 5B7, Canada}
\email{pgupta45@uwo.ca}
\address{Department of Mathematics, University of Western Ontario, London, Ontario N6A 5B7, Canada}
\email{shafikov@uwo.ca}
\begin{abstract} We establish a characterization for an $m$-manifold $M$ to admit $n$ functions $f_1$,...,$f_n$  and $n'$ functions $g_1,...,g_{n'}$ in $\cont^\infty(M)$ so that every element of $\cont^k(M)$ can be approximated by rational combinations of $f_1,...,f_n$ and polynomial combinations of $g_1,...,g_{n'}$. As an application, we show that the optimal value of $n$ and $n'$ for all manifolds of dimension $m$ is $\floor{\frac{3m}{2}}$, when $k\geq 1$ and $m\geq 2$.
\end{abstract}
\maketitle
\section{Introduction}
Let $M$ be a $\cont^\infty$-smooth compact manifold without boundary. We say that $\cont^k(M)$, the space of $k$-times continuously differentiable $\C$-valued functions on $M$, has {\em $n$-rational density} if there is a tuple $F=(f_1,...,f_n)$ of $n$ functions in $\cont^\infty(M)$ such that the set
\bes
	\{R\circ F :R\text{ is a rational function on $\Cn$ with no poles on}\ F(M)\}
\ees
is dense in $\cont^k(M)$ in the  $\cont^k$-norm on $M$. If $F$ exists, we call $\{f_1,...,f_n\}$ an {\em RD-basis} of $\cont^k(M)$. If $F$ can be chosen so that 
\bes
	\{P\circ F :P\text{ is a holomorphic polynomial on}\ \Cn\}
\ees
is $\cont^k$-dense in $\cont^k(M)$, then we say that $\cont^k(M)$ has {\em $n$-polynomial density} and call $\{f_1,...,f_n\}$ a {\em PD-basis} of $\cont^k(M)$. It is a simple consequence of the Stone-Weierstrass theorem that the space of continuous functions on a circle has $1$-rational density and $2$-polynomial density. This paper is motivated by the following two-dimensional analogue of this fact, the first part of which follows from \cite{FoRo93} and the second is proved in \cite{ShSu15}.

\begin{flushleft}
{\em The space of continuous functions on any smooth compact surface has $3$-polynomial density, and with the exception of the $2$-sphere and the projective plane, has $2$-rational density. }
\end{flushleft}
For topological reasons, the space of continuous functions on any surface cannot have $2$-polynomial density. 
We ask whether similar statements can be made for $\cont^k$-spaces on compact manifolds of higher dimensions. To this end, we first obtain the following characterization.

\begin{theorem}\label{thm_main}
Let $k,m,n\in\N$, and $M$ be a compact $\cont^\infty$-smooth $m$-dimensional manifold. 
\begin{enumerate}
\item For $n>m$, $\cont^k(M)$ has $n$-polynomial density if and only if it has $n$-rational density. These conditions are both equivalent to the embeddability of $M$ as a $\cont^\infty$-smooth totally real submanifold of $\Cn$.  
\item $\cont^k(M)$ has $m$-rational density if and only if $M$ admits a $\cont^\infty$-smooth Lagrangian embedding into $(\mathbb C^m,\om\std)$. 
\end{enumerate}
\end{theorem}
The first part of Theorem \ref{thm_main} essentially follows from a result due to L{\o}w and Wold in \cite{WoLo09} that relies on techniques developed by Forstneri{\v c} and Rosay in \cite{FoRo93} and Forstneri{\v c} in \cite{Fo94}. In Section \ref{sec_hprin}, we present a topological approach to the equivalence of $n$-rational density and totally-real embeddability that circumvents the issue of polynomial density.

Theorem \ref{thm_main} allows us to invoke results from the literature to obtain the following generalization of the above-mentioned bounds for surfaces.

\begin{cor}\label{cor_opt} Let $M$ be as in Theorem \ref{thm_main} and $0\leq \ell\leq k$. Then
	\begin{enumerate}
		\item No $\cont^\ell(M)$, $\ell\geq 0$, has $(m-1)$-rational density or $m$-polynomial density.
		\item There is an $n$ with $m\leq n\leq \floor{\frac{3m}{2}}$, such that every  $\cont^\ell(M)$ has $n$-rational density and $n'$-polynomial density, where $n'=\max\{n,m+1\}$. Furthermore, there is an $m$-dimensional $M$ such that no $\cont^\ell(M)$, $\ell\geq 1$, has $\left(\floor{\frac{3m}{2}}-1\right)$-rational density.  
		\end{enumerate} 
\end{cor}

Although the bounds obtained above are optimal in $m$, they can be improved for particular $m$-dimensional manifolds. For instance, any $m$-fold whose complexified tangent bundle is trivializable admits totally real embeddings into $\C^{m+1}$, and can always achieve $(m+1)$-polynomial density. A similar argument (see \cite[Theorem 4.1]{HoJaLa12}) shows that the $\cont^k$-spaces of all orientable manifolds of dimension $4t+2$, $t>0$, have $(6t+2)$-polynomial density, which is an improvement over Corollary \ref{cor_opt}. 

We must also note that the optimality of the upper bound $\floor{\frac{3m}{2}}$ in Corollary \ref{cor_opt} has been stated for $\ell\geq 1$, as it is possible that the optimal bound for $\cont^0$-spaces is lower. This is because we do not have a characterization similar to Theorem \ref{thm_main} for $n$-rational (or polynomial) density of $\cont^0(M)$. For example, if $M$ denotes the double torus, $\cont^0(M)$ has $2$-rational density (by \cite{ShSu15}), but since $\chi(M)$ is nonzero, it does not admit a totally real embedding into $\CC$ (see ~\cite{We69}). For an example in the subcritical case ($n>m$), see Proposition ~\ref{prop_ex}, where we show that if $M$ is the product of two nonorientable surfaces of Euler characteristic $-3$, then $\cont^0(M)$ has $5$-polynomial density, but $M$ does not admit a totally real embedding into $\C^5$. We can say, however, that the optimal bounds for rational density and polynomial density of $\cont^0(M)$ cannot be too far apart due to the following observation.

\begin{cor}\label{cor_basis} Let $M$ be as in Theorem \ref{thm_main} and $0\leq \ell\leq k$. Suppose $\{f_1,...,f_n\}$ is an $RD$-basis of $\cont^\ell(M)$, then there is a $\cont^\infty$-smooth $g:M\rightarrow\C$ such that $\{f_1,...,f_n,g\}$ is a PD-basis of $\cont^\ell(M)$.
\end{cor}

It will be clear from our proof of Theorem \ref{thm_main} that if $F=(f_1,...f_n):M\rightarrow(\Cn,\om\std)$ is an isotropic embedding, then $\{f_1,...,f_n\}$ is an RD-basis. The most practical approach of finding a PD-basis is to seek such an isotropic embedding and to append the $g$ granted by 
the above corollary. For instance, the $n$-torus $\mathbb{T}^n:=\{(z_1,...,z_n):|z_1|=\cdots=|z_n|=1\}$ is Lagrangian with respect to $\om\std$. Thus, $\{z_1,...,z_n\}$ is an RD-basis of $\cont^k(M)$, $k\geq 0$. It is easy to check that $\{z_1,...,z_n,\overline{z_1}\cdot...\cdot\overline{z_n}\}$ forms a PD-basis of $\cont^k(M)$.  

The notion of density is intimately related to that of convexity. A subset $X\subset\Cn$ is called {\em rationally convex} if for every $z\in \Cn\setminus X$, there is a principal algebraic hypersurface that passes through $z$ and completely avoids $X$. If, for each $z\in\Cn\setminus X$ there is a holomorphic polynomial $P$ such that 
	\bes
		|P(z)|>\sup\{|P(x)|:x\in X\},
	\ees
then $X$ is called {\em polynomially convex}. When $\{f_1,...,f_n\}$ is a PD(RD)-basis of $\cont^k(M)$, $F=(f_1,...,f_n)$ maps $M$ onto a polynomially (rationally) convex set in $\Cn$ that is totally real when $k\geq 1$. On the other hand, if $F:M\rightarrow\Cn$ maps onto a polynomially convex set, then, due to an Oka-Weil-type result, functions in $\hol(M)$ can be $\cont^k$-approximated by polynomial functions on $M$. If, furthermore, $F$ is {\em totally real}, i.e., $f_*(T_pM)\cap if_*(T_pM)=0$ for all $p\in M$, then $\hol(M)$ is dense in $\cont^k(M)$. Thus, the question of $n$-polynomial density reduces to that of finding totally real polynomially convex embeddings of $M$ into $\Cn$. But, when $n>m$, L{\o}w and Wold \cite{WoLo09} show that seeking totally real embeddings is enough. This matter has been covered extensively in the literature (see \cite{FoRo93}, \cite{Fo86} and \cite{HoJaLa12}, for instance). 

The second part of Theorem \ref{thm_main} is substantially different as the L{\o}w-Wold argument does not apply to this case ($n=m$). Here, our main tool is a result due to Duval and Sibony (see \cite{DuSi95}) which states that any totally real $\cont^\infty$-smooth submanifold $X\subset\Cn$ is rationally convex if and only if it is {\em isotropic} in $\Cn$ for some K{\"a}hler form $\omega$ --- i.e., $j^*\omega= 0$ on $X$, where $j:X\rightarrow\Cn$ is the inclusion map. 

The rest of this paper is organized as follows. In the next section, we collect some technical lemmas. We present the proof of Theorem \ref{thm_main} in Section \ref{sec_pfthm2}. The corollaries of Theorem \ref{thm_main} are proved in Section \ref{sec_cor}. In Section \ref{sec_hprin}, we discuss an $h$-principle that directly relates rationaly density to totally-real embeddability without any reference to polynomial convexity. In the final section, we summarize the extent of what is known for manifolds of low dimensions, and pose some open questions. 

\smallskip
{\bf Acknowledgments.} We would like to thank Stefan Nemirovski and Alexandre Sukhov for valuable discussions. We are also grateful to Franc Forstneri{\v c} who brought relevant results to our attention that helped us improve an earlier draft of this paper. The second author is 
partially supported by the Natural Sciences and Engineering Research Council of Canada.

\section{Preliminaries}\label{sec_prelim}
We will use the following notation in this paper:

\begin{itemize}
	\item $\floor{x}$ denotes the floor function of $x\in\rl$.
	\item $\hol(X)$ is the space of functions that are holomorphic in some open neighbourhood of a compact $X\subset\Cn$. 
	\item $B(z;r)$ denotes the Euclidean ball of radius $r>0$ centred at $z\in\Cn$. 
	\item $\om\std=dx_1\wedge dy_1+\cdots+dx_n\wedge dy_n$, at any $(z_1,\dots, z_n) = (x_1+iy_1,...,x_n+iy_n)\in\Cn$. 
	\item $R(X)$ is the closed subalgebra of $\cont^0(X)$ that consists of uniform limits of rational functions with no poles on $X$, restricted to $X\subset\Cn$.
	\item $P(X)$ is the closed subalgebra of $\cont^0(X)$ that consists of uniform limits of polynomial functions restricted to $X\subset\Cn$.
\end{itemize}

When talking about $\cont^k$-smooth functions on a $\cont^\infty$-smooth manifold $M$, we keep the following in mind. 

\begin{remark}\label{rmk_norm}
 Let $\mathcal{U}:=\{U_\beta,\phi_\beta\}_{\beta\in B}$ be a finite atlas on $M$. The $\cont^k$-norm on $M$ (with respect to $\mathcal{U}$) is defined as 
	\bes
		||f||_{\cont^k(M),\: \mathcal{U}}=\sum_{\beta\in B}||f\circ\phi_\beta^{-1}||_{\cont^k(\phi_\beta(U_\beta))}.
	\ees
Different choices of $\mathcal{U}$ lead to different, but equivalent, norms on $\cont^k(M)$. The definition of $n$-rational density is, therefore, independent of the choice of $\mathcal{U}$. It follows that if $F:M\rightarrow M'$ is a $\cont^k$-smooth diffeomorphism, then the $\cont^k$-convergence of a sequence $\{h_j\}_{j\in\N}\subset \cont^k(M')$ to some $h\in\cont^k(M')$ is equivalent to the $\cont^k$-convergence of $\{h_j\circ F\}_{j\in\N}$ to $h\circ F$.
\end{remark}

We now present a fact that is implicitly proved in \cite{DuSi95}.

\begin{lemma}\label{lem_ratnbd} Let $S$ be a $\cont^k$-smooth, $k\geq 1$, totally real compact submanifold of $\Cn$ that is rationally convex. Then, for any open neighbourhood $U$ of $S$, there is an open neighbourhood $V\Subset U$ of $S$ such that $\overline V$ is rationally convex.
\end{lemma}
\begin{proof}  Let $S$ and $U$ be as given. Since $S$ is totally real, we can rely on Theorem 6.1.6 in \cite{St07}, to obtain (after shrinking $U$, if need be) a non-negative strictly plurisubharmonic function $\rho:U\rightarrow \rl$ such that $\rho\in\cont^\infty(U\setminus S)\cap\cont^{k+1}(U)$ and $\rho^{-1}(\{0\})\cap U=S$. Next, we fix a neighbourhood $V\Subset U$ of $S$, and choose $\chi$ to be a smooth cut-off function such that
	\beas 
    	&&\chi\big|_V\equiv 1;\\
        &&\chi\big|_{\Cn\setminus U}\equiv 0.
     \eeas
We also fix a closed origin-centred ball $B\subset\Cn$ containing $U$. Our goal is to invoke the following consequence of Theorem 2.1 in \cite{DuSi95}: {\em Let $S$ be a rationally convex compact set in $\Cn$. For every $x\notin S$, there is a smooth positive closed $(1,1)$-form that is strictly positive at $x$ and vanishes in a neighbourhood of $S$.}

In a remark following the above theorem, Duval and Sibony show that, in fact, one can construct a smooth closed $(1,1)$-form which is strictly positive on $\Cn\setminus S$, vanishes on $S$ and has a global potential that is $C'|z|^2$, for some $C'>0$, outside a ball containing $S$. In view of this, we let $\omega_B$ be a closed $(1,1)$-form that vanishes on $B$, is strictly positive outside $B$ and has a global potential that is $C'|z|^2$ outside a larger ball $B'$. For any $z\in B\setminus V $, we can find a smooth positive closed $(1,1)$-form $\omega_z$ that vanishes in some neighbourhood $W_z$ of $S$, is positive in some neighbourhood $N_z$ of $z$ and has a global potential that is $C'|z|^2$ outside $B'$. We cover the compact set $B\setminus V$ by finitely many such neighbourhoods $N_{z_1},...,N_{z_k}$, and set 
	\bes 
    	W:=W_{z_1}\cap\cdots\cap W_{z_k}\cap V. 
    \ees
Note that $W$ is a neighbourhood of $S$ contained in $V$. Now, for sufficiently large constants $c_j>0$, 
	\bes
    	\tilde\omega:=dd^c(\chi\rho)+\omega_B+\sum_{j=1}^kc_j\omega_{z_j},
    \ees
is a K{\"a}hler form on $\Cn$ with $\tilde\omega\big|_W=dd^c\rho$, and potential $C|z|^2$ outside $B'$, for some $C>0$. 

Now, we can employ the following result of Nemirovski (see \cite[Prop. 1]{Ne08}): {\em  Let $\phi$ be a strictly plurisubharmonic function on an open subset $U\subset\Cn$
 such that its Levi form $dd^c\phi$ extends to a positive $d$-closed $(1,1)$-form on $\Cn$. If
the set $K_\phi = \{z \in U | \phi(z)\leq  0\}$ is compact, then it is rationally convex.} We can choose $c>0$ small enough so that $S_c:=\{z\in U:\rho(z)<c\}\Subset W$. Then, $\tilde\omega$ is 
the required extension of $dd^c(\rho-c)$, and $S_c$ is the required neighbourhood of $S$. 
\end{proof}

The next lemma is an application of Moser's trick that allows us to change the underlying K{\"a}hler form.

\begin{lemma}\label{lem_moser} Let $\om=dd^c\phi$ for some $\cont^\infty$-smooth strictly plurisubharmonic function $\phi$ on $\Cn$. If a manifold $M$ admits a $\cont^\infty$-smooth isotropic embedding into $(\Cn,\om)$, then it admits a $\cont^\infty$-smooth isotropic embedding into $(\Cn,\om\std)$.
\end{lemma}
\begin{proof} We abuse notation and let $M$ be an isotropic submanifold of $(\Cn,\om)$. We may further assume that for some $C>0$, $\phi=C|z|^2+d$ outside a ball containing $M$. To see this, consider positive integers $r_1<r_2<r_3<r_4$ so that $M\subset B(0;r_1)$. Let $\eta:\Cn\rightarrow\rl$ be a $\cont^\infty$-smooth cutoff function such that
	\beas	
		\eta(z)=\begin{cases}
			1, &\ \text{when}\ |z|\leq r_3;\\
			0, & \ \text{when}\ |z|\geq r_4.
		\end{cases}
		\eeas
Next, we choose $C,d>0$ so that 
	\begin{itemize}
\item $\eta(z)\phi(z)+C|z|^2$ is strictly plurisubharmonic on $\Cn$; and
\item there is a nondecreasing, convex $\cont^\infty$-function $\theta:\rl\rightarrow\rl$ such that
	\beas 
		\theta(x)=
			\begin{cases}
			0, &\ \text{when}\ |x|\leq r_1;\\
			x+d, &\ \text{when}\ |x|\geq r_2.
			\end{cases}
	\eeas 
\end{itemize}
Then, relabelling $\phi$ to be the strictly plurisubharmonic function 
	\bes
		z\mapsto\eta(z)\phi(z)+\theta(C|z|^2),
	\ees
we note that $M$ is isotropic with respect to $\om=dd^c\phi$, since the above modification leaves $\phi$ unchanged in $B(0;r_1)$. 

Now, let $\phi_t=tC|z|^2+(1-t)\phi$. Then, $\om_t:=dd^c\phi_t=dd^c\phi+d\beta_t$, where $\beta_t=d^c(t\phi_1-t\phi)$. Note that $\beta_t\equiv 0$ outside $B(0;r_4)$. Consider the smooth (in $z$) compactly supported vector field  $X_t$ defined by 
	\bes
		\iota_{X_t}\om_t+\partl{}{t}\beta_t=0,
	\ees
where $\iota_X\om$ denotes the contraction of $\om$ with $X$. $X_t$ is unique since each $\om_t$ is nondegenerate (each $\phi_t$ is strictly plurisubharmonic). As $X_t$ is compactly supported, its flow $\Phi_t$ exists. Moreover,
	\beas
		\partl{}{t} \Phi_t^*\om_t&=&\Phi_t^*\left(\partl{}{t}\om_t+L_{X_t}\om_t\right)\\		
		&=&
		\Phi_t^*\left(\partl{}{t}d\beta_t+\iota_{X_t}d\om_t+d\iota_{X_t}\om_t\right)\\
		&=&
		\Phi_t^*d\left(\partl{}{t}\beta_t+\iota_{X_t}\om_t\right)=0.
	\eeas
So, $\Phi_1:\Cn\rightarrow\Cn$ is a $\cont^{\infty}$-diffeomorphism such that $\Phi_1^*(4C\om\std)=dd^c\phi=\om$. Thus, $\Phi_1(M)$ gives the required isotropic embedding of $M$ into $(\Cn,\om\std)$. 
\end{proof}

\section{Proof of Theorem \ref{thm_main}}\label{sec_pfthm2}
Let us consider the case when $n>m$. We will first show that the embeddability of $M$ as a $\cont^\infty$-smooth totally real submanifold of $M$ implies that $\cont^k(M)$ has $n$-polynomial density. As $n$-polynomial density trivially implies $n$-rational density, it will then suffice to show that if $\cont^k(M)$ has $n$-rational density, $M$ admits a totally real $\cont^\infty$-embedding into $\Cn$. 

Suppose $M$ admits a $\cont^\infty$-embedding 
$F:M\rightarrow\Cn$ so that $F(M)$ is totally real. Due to Theorems $1$ and $2$ in \cite{WoLo09}, there is a $\cont^\infty$-smooth map $G:M\rightarrow\Cn$ ($\cont^1$-close to $F$) such that $G(M)$ is totally real and polynomially convex in $\Cn$. Set $M':=G(M)$.

Next, we fix an $f\in\cont^k(M')$ and choose an arbitrary $\eps>0$. Let $\wt\eps=C\eps$, where $C$ is a constant to be determined later. Since $M'$ is totally real, a result due to Range and Siu (see Theorem 1 in \cite{RaSi74}) grants the existence of a neighbourhood $U$ of $M'$ and a $g\in\hol(U)$ such that
	\be\label{eq_c-h}
		||f-g||_{\cont^k(M')}<\wt\eps,
	\ee
Due to the polynomial convexity of $M'$, we can find a neighbourhood $V\Subset U$ of $M'$, such that $\overline V$ is polynomially convex. By the Oka-Weil approximation theorem for polynomially convex sets, there is a polynomial function $P$ on $\Cn$ such that  
	\be\label{eq_h-r_sup}
		||g-P||_{\cont^0(\overline V)}< \wt\eps.
	\ee
Now, fix an $x\in M'$. Since $M'$ is compact, there is an $r>0$ --- independent of $x\in M'$ --- such that $B(x;r)\subset V$. As $g-P\in\hol(B(x,r))$, we can combine Cauchy estimates and \eqref{eq_h-r_sup} to obtain
	\bea 
		\left|g^{(j)}(x)-P^{(j)}(x)\right|
			&\leq&\frac{j!}{r^j}\sup_{y\in B(x;r)}|g(y)-P(y)| \notag \\
			&\leq &\frac{j!}{r^j}||g-P||_{\cont^0(\overline V)}
			\leq \frac{j!}{r^j}\wt\eps,		\label{eq_h-r_ck}								
	\eea
for any $j\in\pnat$. So, we obtain from \eqref{eq_c-h} and \eqref{eq_h-r_ck} that
	\beas
		||f-P||_{\cont^k(M')}&\leq &||f-g||_{\cont^k(M')}+||g-P||_{\cont^k(M')}\\
		&=&||f-g||_{\cont^k(M')}+\sum_{j=0}^k ||g^{(j)}-P^{(j)}||_{\cont^0(M')}\\
		&<&\wt \eps\left(1+\sum_{j=0}^k \frac{j!}{r^j}\right)
		=C\eps\left(1+\sum_{j=0}^k \frac{j!}{r^j}\right).
	\eeas
Now, setting $C=\left(1+\sum_{j=0}^k \frac{j!}{r^j}\right)^{-1}$, we obtain that 	
\bes 
	||f-P||_{\cont^k(M')}<\eps.
\ees
Since $\eps>0$ and $f\in\cont^k(M')$ were chosen arbitrarily, and $C$ is independent of $\eps$, we conclude that rational functions are dense in the space of $\cont^k$-smooth functions on $M'$ in the $\cont^k$-norm. In view of Remark ~\ref{rmk_norm}, the $n$ components of $G$ form a PD-basis (and RD-basis) of $\cont^k(M)$.

We now complete the proof of the subcritical case of Theorem \ref{thm_main}. Let $M$ be a $\cont^\infty$-smooth manifold such that $\cont^k(M)$ has $n$-rational density, and $\{f_1,...,f_n\}$ is an RD-basis of $\cont^k(M)$. We claim that $F:M\rightarrow\Cn$ given by $F=(f_1,...,f_n)$ is a $\cont^\infty$-smooth totally real embedding of $M$ into $\Cn$. 

Suppose $F$ is not injective --- i.e., there are two distinct points $p_1, p_2\in M$ such that $f_j(p_1)=f_j(p_2)$ for all $j=1,...,n$. Then, for any function $G:F(M)\rightarrow\C$, $(G\circ F)(p_1)=(G\circ F)(p_2)$. Now, by the $n$-rational density of $\cont^k(M)$, any $g$ in $\cont^k(M)$ can be uniformly approximated on $M$ by functions of the form $R\circ F$, where $R$ is a rational function on $\Cn$ with no poles on $F(M)$. Hence, $g(p_1)=g(p_2)$ for any $g\in\cont^k(M)$. But this is a contradiction as $\cont^k(M)$ separates points.  

Next, we show that $F$ is a local embedding at every point in $M$. For this, we fix a $p\in M$, and choose a coordinate chart $\phi:U\rightarrow\phi(U)\subset\rl^m$ around $p$, with $\phi(p)=0$. If $\phi=(x_1,...,x_m)$, we write, in local coordinates,
	\bes
		F_{x_j}(x):=\left(\partl{f_1}{x_j}(x),...,\partl{f_n}{x_j}(x)\right),\qquad j=1,...,m.
	\ees
Suppose $\op{rank}(dF(0))<m$. Then, up to relabelling, there exist $\C$-valued constants $\alpha_1,...,\alpha_{m-1}$ such that
	\bes
		\alpha_1F_{x_1}(0)+\cdots+\alpha_{m-1}F_{x_{m-1}}(0)=F_{x_m}(0).
	\ees
Therefore, if $R$ is a holomorphic map in some neighbourhood of $F(U)$, then
	\beas
		\sum_{j=1}^{m-1}\alpha_j\partl{(R\circ F)}{x_j}(0)
		&=&\sum_{j=1}^{m-1}\alpha_j\left(\sum_{k=1}^{n}\partl{R}{z_k}(F(0))\partl{f_k}{x_j}(0)\right)\\
		&=&\sum_{k=1}^{n}\partl{R}{z_k}(F(0))\left(\sum_{j=1}^{m-1}\alpha_j\partl{f_k}{x_j}(0)\right)\\
		&=&\sum_{k=1}^{n}\partl{R}{z_k}(F(0))\partl{f_k}{x_m}(0)\\
		&=&\partl{(R\circ F)}{x_m}(0).
	\eeas
It follows that every function $g$ in the set (and therefore, in the $\cont^k$-closure of the set) 
	\bes
		R_F(M):=\{R\circ F:R\ \text{is a rational function on $\Cn$ with no poles on}\ F(M)\}
	\ees
has the property that 
	\bes
	\partl{g}{x_m}(0)=\sum_{1}^{m-1}\alpha_j\partl{g}{x_j}(0).
	\ees 
But, by our assumption, the $\cont^k$-closure of $R_F(M)$ is $\cont^k(M)$, and one can easily construct a $\cont^k$-smooth function $g$ on $M$ such that $\partl{g}{x_j}(0)=0$ for $j=1,...,m-1$, but $\partl{g}{x_m}(0)\neq 0$. Thus, $\op{rank}(dF(p))=m$ for every $p\in M$, and $F:M\rightarrow\Cn$ is a $\cont^k$-smooth embedding.

It now remains to show that $F(M)$ is a totally real submanifold of $\Cn$. This is done using the same technique as in the preceding paragraphs. Suppose $p\in M$ is such that $T_pF(M)\cap iT_pF(M)\neq 0$. Then, we can choose local coordinates $x_1,...,x_m$ around $p$ such that 
	\bes
		i\partl{F}{x_1}(p)=\partl{F}{x_2}(p).
	\ees
This property is inherited by any function of the form $R\circ F:M\rightarrow\C$, where $R\in\hol(F(M))$, and, therefore, by any function that can be expressed as a $\cont^1$-limit of functions of this form. Hence, by our hypothesis on $M$, any $\cont^k$-smooth function $g$ on $M$ satisfies
		\bes
		i\partl{g}{x_1}(p)=\partl{g}{x_2}(p).
	\ees 
This is a contradiction as one can construct $\cont^k$-smooth functions on $M$ with any prescribed $1$-jet at $p$. Thus, $F(M)$ is a $\cont^\infty$-smooth totally real submanifold of $\Cn$.

For the first part of the case $n=m$, we assume that $\cont^k(M)$ has $m$-rational density (with RD-basis $\{f_1,...,f_m\}$) and repeat the argument above to conclude that $F(M)$ is a $\cont^\infty$-smooth totally real submanifold of $\C^m$, where $F=(f_1,...,f_m)$. Moreover, $F(M)$ is rationally convex. To see this, recall that any continuous function on $F(M)$ can be uniformly approximated by $\cont^k$-functions on $F(M)$. But, $R(F(M))=\cont^k(F(M))$ (see the beginning of Section \ref{sec_prelim} for notation). Thus, $R(F(M))=\cont^0(F(M))$. This is a sufficient condition for rational convexity. So, $F(M)$ is a totally real and rationally convex $\cont^\infty$-submanifold of $\Cn$. By Theorem~3.1 in \cite{DuSi95}, $F(M)$ is Lagrangian with respect to some smooth K{\"a}hler form $\om=dd^c\phi$. An application of Lemma~\ref{lem_moser} shows that $M$ admits a Lagrangian embedding into $(\Cn,\om\std)$.   

Conversely, suppose $M$ admits a $\cont^\infty$-Lagrangian embedding with respect to $\om\std$. By \cite[Theorem 3.1]{DuSi95}, $F(M)$ is rationally convex. We can now repeat the first portion of our proof for the subcritical case to conclude that the components of $F$ form an RD-basis of $\cont^k(M)$. We need two ingredients for this: the existence of arbitrarily small rationally convex neighborhoods of $F(M)$, and an Oka-Weil-type result for rationally convex sets. The former is granted by Lemma ~\ref{lem_ratnbd}, and the latter is a standard result (see \cite[pg. 44]{St07}, for instance). The proof of Theorem ~\ref{thm_main} is now complete.
\begin{remark}
As any $\cont^\infty$-smooth embedding of $M$ can be $\cont^1$-approximated by a real-analytic one, we can apply Theorem~1 from \cite{WoLo09} (and an analogous theorem for rationally convex sets) to conclude that when $k>0$, any PD-basis (or RD-basis) of $\cont^k(M)$ can be perturbed to obtain a real-analytic basis. 
\end{remark}

\section{Proofs of the corollaries}\label{sec_cor}
\begin{proof}[Proof of Corollary ~\ref{cor_opt}.] $(1)$ As seen in the previous section, $(m-1)$-RD bases and $m$-PD bases of $\cont^\ell(M)$ ($0\leq \ell\leq k$) yield rationally and polynomially convex embeddings of $M$ into $\C^{m-1}$ and $\C^m$, respectively. This is topologically impossibe as no $m$-dimensional manifold can be rationally convex in $\C^{m-1}$ (see \cite[Corollary 2.3.10]{St07}) or polynomially convex in $\C^m$ (see \cite[Corollary 2.3.5]{St07}). 

$(2)$ Owing to Forstneri{\v c} and Rosay (\cite{FoRo93}), it is known that any $m$-dimensional manifold admits a totally real embedding into $\Cn$ if $n=\floor{\frac{3m}{2}}$. When $m\geq 2$, $m<\floor{\frac{3m}{2}}$, so by Part $1$ of Theorem ~\ref{thm_main}, the claim follows. When $m=1$, the only compact manifold without boundary is the circle, for which the result is standard.  

For the optimality of the bound, we consider the examples produced in \cite{HoJaLa12}. Let $M^{4t}:=\CP^2\times\cdots\times\CP^2$ be the product of $t$ copies of the complex projective plane. Then, invoking Theorem 2.1 from \cite{HoJaLa12}, we have that for any $\ell>0$,
	\beas
		&&\cont^\ell(M^{4t})\ \text{does not have $(6t-1)$-polynomial density};\\
		&&\cont^\ell(M^{4t}\times S^1)\ \text{does not have $6t$-polynomial density};\\
		&&\cont^\ell(M^{4t}\times\RP^2)\ \text{does not have $(6t+2)$-polynomial density};\\
		&&\cont^\ell(M^{4t}\times\RP^2\times S^1)\ \text{does not have $(6t+3)$-polynomial density}.
	\eeas
\end{proof}

For the next proof, we rely on the fact that compact rationally convex sets in $\Cn$ can be realized as polynomially convex subsets of $\C^{n+1}$. 

\begin{proof}[Proof of Corollary ~\ref{cor_basis}.] 
We first assume that $\{f_1,...,f_n\}$ is an RD-basis of $\cont^0(M)$. As noted previously, $F:(f_1,...,f_n)$ maps $M$ onto a rationally convex subset $M'$ of $\Cn$ such that $R(M')=\cont^0(M')$. By the proof of Theorem 1.2.11 in \cite{St07} (due to Rossi and Basener), there is a $\psi\in\cont^\infty(M')$ such that the closed subalgebra $[P(M'),\psi]$ generated by the polynomials and $\psi$ on $M'$ coincides with $R(M')$. Thus, if $\wt F:=(f_1,...,f_n,g):M\rightarrow\C^{n+1}$, where $g:=\psi\circ F$, $P(\wt F(M))=R(M')=\cont^0(M')=\cont^0(M)$. But this is precisely what it means for $\{f_1,...,f_n,g\}$ to be a PD-basis of $\cont^0(M)$. 

Now, suppose $\{f_1,...,f_n\}$ is an RD-basis of $\cont^\ell(M)$, where $\ell\geq 1$. Then, $M'=F(M)$ is a totally real and rationally convex submanifold of $\Cn$ (as argued in the proof of Theorem \ref{thm_main}). Again, we let $\psi:M'\rightarrow\C$ denote a $\cont^\infty$-smooth function such that the graph $\Gamma_\psi:=\{(x,\psi(x))\in\C^{n+1}:x\in M'\}$ is polynomially convex. We claim that $\Gamma_\psi$ is totally real. Let $\iota_n:M'\hookrightarrow \Cn$ and $\iota_{n+1}:\Gamma_g\hookrightarrow \C^{n+1}$ denote inclusions maps and $\Psi:\Cn\rightarrow\C^{n+1}$ be defined as $z\mapsto (z,\psi(z))$. Note that $\iota_{n+1}=\Psi\circ\iota_n\circ\pi$ on $\Gamma_\psi$, where $\pi:\Gamma_\psi\rightarrow M'$ is the projection map onto the first $n$ coordinates. Then,
	\beas	
		\iota_{n+1}^*(dw_1\wedge\cdots\wedge dw_n)&=&(\Psi\circ\iota_n\circ\pi)^*(dw_1\wedge\cdots\wedge dw_n)\\
														&=&\pi^*\iota_n^*(dz_1\wedge\cdots \wedge dz_n),
	\eeas 
which is nonzero everywhere on $\Gamma_\psi$ as $\iota_n^*(dz_1\wedge\cdots \wedge dz_n)$ is nonzero on $M'$ ($M'$ is totally real) and $(\pi_*)_p:T_{\Psi(p)}\Gamma_\psi\rightarrow T_pM'$ is a $\C$-linear isomorphism for every $p\in M'$. Thus, $\Gamma_\psi$ is a polynomially convex and totally real smooth subset of $\C^{n+1}$. In Section \ref{sec_pfthm2}, we saw that this suffices to conclude that $\{f_1,...,f_n,g=\psi\circ F\}$ is a PD-basis of $\cont^\ell(M)$. 
\end{proof}

\section{An \texorpdfstring{$\mathrm{h}$}{h}-Principle}\label{sec_hprin}
In our proof of Part $1$ of Theorem \ref{thm_main}, we rely on a powerful result from \cite{WoLo09} that establishes the genericity of polynomially convex embeddings of $M$ in $\Cn$ among totally real ones. As rational convexity is weaker than polynomial convexity, one can establish the equivalence of $n$-rational density and totally-real embeddability into $\Cn$ using purely topological methods without appealing to the constructive methods in \cite{WoLo09}, \cite{FoRo93} and \cite{Fo94}. For this, we establish a link between the totally-real embeddability and the isotropic embeddability of any $m$-dimensional manifold $M$ into $\Cn$. When $n=m$, the Gromov-Lees theorem (see \cite{AuLaPo94}) states that $M$ admits a Lagrangian immersion into $(\Cn,\om\std)$ precisely when its complexified tangent bundle is trivializable. This is the same topological condition that completely characterizes the totally real immersability of $M$ in $\Cn$ (see \cite[Prop. 9.1.4]{Fo11}). It is, therefore, not surprising that when $n>m$, the obstruction to the existence of totally real embeddings and to that of isotropic embeddings is one and the same. Although there is some indication of this in the literature, due to the lack of a clear reference, we provide a complete statement and proof of this fact.   

\begin{lemma}\label{lem_tr-iso} Let $M$ be a $\cont^k$-smooth $(k\geq 1)$ compact manifold of real dimension $m$, and $m<n$. Then, $M$ admits a $\cont^k$-smooth totally real embedding in $\Cn$ if and only if $M$ embeds in $(\Cn,\om\std)$ as a $\cont^\infty$-smooth isotropic submanifold via a $\cont^k$-embedding. 
\end{lemma}
\begin{proof} The main tool in this proof will be the $h$-principle for contact isotropic immersions. For this purpose, we first transfer the problem to the $\cont^\infty$-smooth category. Let $\tau:M\rightarrow\Cn$ be a $\cont^k$-smooth totally real embedding of $M$ into $\Cn$. Then, by standard approximation results, there exists a $\cont^k$-diffeomorphism $\phi:\Cn\rightarrow\Cn$ so that $\phi(\tau(M))$ is a $\cont^\infty$-smooth (embedded) submanifold of $\Cn$ and $\phi$ is $\cont^k$-close to the identity. Due to the latter property, $\phi$ can be chosen so that $\phi(\tau(M))$ continues to be a totally real submanifold of $\Cn$. This is because the condition of being totally real is an open one. We relabel $M$ to denote $\phi(\tau(M))$ and note that $M$ is now a $\cont^\infty$-smooth totally real submanifold of $(\Cn,\om\std)$. 

Now, consider the contact manifold $(\Cn\times\rl,\alpha\std)$, where
	\bes
		(\alpha\std)_{(z,t)}=dt-\sum_{ j=1}^ny_jdx_j,
	\ees
$z=(x_1+iy_1,...,x_n+iy_n)\in\Cn$ and $t\in\rl$. Let $\xi\std$ denote the corresponding contact structure --- i.e., the codimension one sub-bundle of the tangent bundle $T(\Cn\times\rl)$ that has fibre $(\xi\std)_p=\ker(\alpha\std)_p$ at each $p\in\Cn\times\rl$. We observe that if $\pi:\Cn\times\rl\rightarrow\Cn$ denotes the projection map, then 
	\be\label{eq_symcon}
		\pi^*(\om\std)=d\alpha\std.
	\ee
Also, for each $p\in\Cn\times\rl$, there is an $\rl$-isomorphism $T_{\pi(p)}\Cn\cong(\xi\std)_p$ given by 
\beas
	\lambda_p:e_{x_j}&\mapsto& (e_{x_j},y_j(p))\\
	\lambda_p:e_{y_j}&\mapsto& (e_{y_j},0),
\eeas
where $\{e_{x_1},e_{y_1},...,e_{x_n},e_{y_n}\}$ is the standard basis of $\rl^{2n}$, and $y_j(p)$ denotes the $y_j$-th coordinate of $p$. As $\lambda_p$ is smooth in $p$, it gives a (real) bundle morphism $\lambda:T\Cn\rightarrow \xi\std$ covering the inclusion map $j:\Cn\rightarrow\Cn\times\rl$ given by $z\mapsto (z,0)$. Moreover,
	\be\label{eq_consym}
		\lambda^*(d\alpha\std)=\om\std.
	\ee
We will use \eqref{eq_symcon} and \eqref{eq_consym} to move between $(\Cn,\om\std)$ and $(\Cn\times\rl,\alpha\std)$. 

Recall that an immersion $f:N\rightarrow\Cn\times\rl$ is called contact isotropic if $df:TN\rightarrow\xi\std$. Moreover, 
	\bes 	
		f^*(d\alpha\std)=d(f^*\alpha\std)=0.
	\ees
From \eqref{eq_symcon}, we have that
	\bes
		(\pi\circ f)^*(\om\std)=f^*(\pi^*(\om\std))=f^*(d\alpha\std)=0.
	\ees
Thus, $\pi\circ f:N\rightarrow\Cn$ is a symplectic isotropic immersion. In general, $\pi\circ f$ need not be an immersion if $f$ is merely an immersion, but here we have the additional fact that $(f_*)_p(T_pN)\cap \ker(\pi_*)_p=0$ for every $p\in N$, since $f_*(TN)\subset\ker(\alpha\std)$ and $\ker(\alpha\std)\cap\ker\pi_*=0$. To summarize, 
\begin{center}\refstepcounter{equation}(\theequation)\label{eq_proj}
{\em every contact isotropic immersed manifold in $(\Cn\times\rl,\alpha\std)$ projects to a symplectic isotropic immersed manifold in $(\Cn,\om\std)$}. 
\end{center}

Now, suppose we have a sequence of maps
	\bes
		N\xrightarrow{g} \Cn\xrightarrow{j} \Cn\times\rl,
	\ees
where $j$ is the inclusion map defined above. A contact isotropic monomorpism covering $j\circ g$ is a monomorphism $F:TN\rightarrow\xi\std$ covering $j\circ g$ such that $F^*(d\alpha\std)=0$. On the other hand, a symplectic isotropic monomorphism $G:TN\rightarrow T\Cn$ covering $g$ is one that satisfies $G^*(\om\std)=0$. Combining this with \eqref{eq_consym}, we have that
	\bes
		(\lambda\circ G)^*(d\alpha\std)=G^*(\lambda^*(d\alpha\std))=G^*(\om\std)=0.
	\ees
Hence, we have that 
\begin{center}\refstepcounter{equation}(\theequation)\label{eq_lift}
{\em every symplectic isotropic monomorphism $G:TN\rightarrow T\Cn$ covering $g$ lifts to a contact isotropic monomorphism $\lambda\circ G:TN\rightarrow\xi\std$ covering $j\circ g$}. 
\end{center}

Let us now state the relevant $h$-principle. Let $\op{Imm}_{\op{isotr}}(M;\Cn\times\rl,\alpha\std)$ be the space of contact isotropic immersions of $M$ into $(\Cn\times\rl,\alpha\std)$, and $\op{Mon}_{\op{isotr}}(TM;\xi\std)$ be the space of contact isotropic monomorphisms $TM\rightarrow\xi\std$. Note that any element $g\in\op{Imm}_{\op{isotr}}(M;\Cn\times\rl,\alpha\std)$ can be identified with $dg\in\op{Mon}_{\op{isotr}}(TM;\xi\std)$. Thus, there is an inclusion map 
	\be\label{eq:hprin}
		\op{Imm}_{\op{isotr}}(M;\Cn\times\rl,\alpha\std)
		\hookrightarrow 
		\op{Mon}_{\op{isotr}}(TM;\xi\std).
	\ee
The $h$-principle in this context states that the inclusion \eqref{eq:hprin} is, in fact, a homotopy equivalence (see \cite[Theorem ~7.9]{CeEl12}). So, $M$ admits a contact isotropic immersion into $\Cn\times\rl$ if $\op{Mon}_{\op{isotr}}(TM;\xi\std)$ is nonempty. Due to the observation \eqref{eq_proj}, the nonemptiness of $\op{Mon}_{\op{isotr}}(TM;\xi\std)$ would imply that $M$ admits a smooth symplectic isotropic immersion into $\Cn$. 

We now produce an element in $\op{Mon}_{\op{isotr}}(TM;\xi\std)$. Let $\iota:M\rightarrow\Cn$ denote the inclusion map. Then, $d\iota:TM\rightarrow T\Cn$ is a totally real monomorphism covering $\iota$. We complexify this map to obtain a complex-linear monomorphism $d_\C\iota:\C TM\rightarrow T\Cn$ in the following way:
	\bes
		d_{\C}\iota:(p,v\otimes (a+ib))\mapsto\big(p,a\:d\iota_p(v)+ib\:d\iota_p(v)\big).
	\ees
Now, for any complex $m$-frame $f_p:\C^{m}\rightarrow\C T_pM$, we obtain an injective $\C$-linear map $d_\C\iota_p\circ f_p:\C^m\rightarrow T_{\iota(p)}\Cn=\Cn$. Thus, $d\iota$ lifts to a map on the frame bundle associated to $\C TM$:
	\bes
		\iota_{_{\mathcal{F}}}:\mathcal{F}_\C M\rightarrow \Cn\times V_{n,m},
	\ees
where $V_{n,m}$ is the Stiefel manifold of all complex $m$-frames in $\Cn$. By the Gram-Schmidt orthogonalization process, $V_{n,m}$ is the product of $U_{n,m}$ --- the space of unitary $m$-frames in $\Cn$, and $P$ --- the space of upper-triangular matrices with positive eigenvalues. As $P$ is contractible, there is a smooth homotopy $H:\mathcal{F}_\C M\times[0,1]\rightarrow \Cn\times V_{n,m}$ such that $H(\cdot,0)=\iota_{_{\mathcal{F}}}(\cdot)$, $H(\cdot,1)$ maps $\mathcal{F}_\C M$ into $\Cn\times U_{n,m}$ and $H(\cdot,t)$ covers $\iota$ for all $t\in[0,1]$. This homotopy descends to $TM$ to give monomorphisms $F^s:TM\rightarrow T\Cn$ covering $\iota$, such that $F^0=d\iota$. Moreover, since the real linear span of a unitary $m$-frame in $\Cn$ is an $m$-dimensional isotropic subspace (with respect to $\om\std$), $F^1$ is a symplectic isotropic monomorphism. By \eqref{eq_lift}, $\lambda\circ F^1\in\op{Mon}_{\op{isotr}}(TM;\xi\std)$. Thus, there is a smooth immersion $f:M\rightarrow\Cn$ so that $f(M)$ is symplectic isotropic in $\Cn$. Since $\dim M<n$, a general position lemma allows us to approximate $f$ by an isotropic embedding $g$ (see \cite[Lemma~1.2.4]{AuLaPo94}). Recall that $M$ is actually $\phi(\tau(M))$, where $\phi$ and $\tau$ are both $\cont^k$-smooth. The embedding we seek is $g\circ\phi\circ\tau$, which is clearly $\cont^k$-smooth. We have proved one direction of our claim. 

Any isotropic linear subspace of $\Cn$ is necessarily a totally real one. So, the converse statement poses no challenge, and the proof of Lemma~\ref{lem_tr-iso} is complete.
\end{proof}

\begin{remark} It is likely that the contactification procedure can be avoided in the above proof. One approach would be to directly invoke an $h$-principle for symplectic isotropic immersions. Such an $h$-principle is alluded to in \cite[Section~14.1]{ElMi02}, but an actual statement is missing. Alternatively, we could invoke the $h$-principle for subcritical embeddings stated in \cite[Section~12.4]{ElMi02}, but it is missing the case $n=m+1$ (the required hypothesis is $m<\floor{\frac{\dim_\rl\Cn-1}{2}}=n-1$). We attribute this to an oversight, as the proof given therein seems to work also for $m\leq\floor{\frac{\dim_\rl\Cn-1}{2}}$ in the symplectic scenario.
\end{remark}

\section{Examples and Open Questions}\label{sec_ex}	
We first show that in contrast to Theorem \ref{thm_main}, $n$-polynomial density of $\cont^0(M)$ does not have a characterization in terms of totally real embeddings (or immersions) of $M$ into $\Cn$.

\begin{Prop}\label{prop_ex}
Let $M:=S_{_{-3}}\times S_{_{-3}}$, where $S_{_{-3}}:=(\RP^2)^{\# 5}$ is the connected sum of $5$ projective planes. Then, $\cont^0(M)$ has $5$-polynomial density, but $M$ does not admit a totally real immersion into $\C^5$.
\end{Prop}
\begin{proof} Note that by Corollary \ref{cor_basis}, it is enough to show that $\cont^0(M)$ has $4$-rational density. For this, we first use a result due to Nemirovski and Siegel (see \cite{NeSi16}) which says that a nonorientable surface with $\chi\leq -1$ admits a Lagrangian embedding in $\CC$ with $k$ open Whitney umbrellas if and only if 	
	\bes
		k\in\{4-3\chi, -3\chi,-3\chi-4,...,\chi+4-4\floor{\chi/4+1}\}.
	\ees
Since $\chi(S_{_{-3}})=-3$, we can embed $S_{_{-3}}$ into $\CC$ as a Lagrangian surface with one open Whitney umbrella (say at the origin). Let us denote the embedded $S_{_{-3}}$ by $S$. In \cite{ShSu15}, it is shown that $S$ is holomorphically and rationally convex, and $R(S)=\cont^0(S)$. It follows that $X=S\times S$ is a topological embedding of $M$ in $\C^4$ that is holomorphically and rationally convex, and totally real away from $X_0=\big(S\times\{0\}\big)\cup \big(\{0\}\times S\big)$. The rational convexity of $X$ implies that $R(X)=\overline{\hol(X)}$, where $\overline{\hol(X)}$ denotes the closed subalgebra of $\cont^0(X)$ that consists of uniform limits on $X$ of elements in $\hol(X)$. It, therefore, suffices to show that $\cont^0(X)=\overline{\hol(X)}$. For this, choose an arbitrary $f\in\cont^0(X)$. We claim that $X$, $X_0$ and $f$ satisfy the conditions of the following result by O'Farrel, Preskenis and Walsch \cite{OFPrWa84}:
\smallskip

\noindent {\em  Let $X\subset\Cn$ be a compact holomorphically convex set, and let $X_0$ be a closed subset of $X$ such that $X\setminus X_0$ is a totally real subset of $\Cn\setminus X_0$. A function $f\in\cont^0(X)$ can be approximated uniformly on $X$ by functions in $\hol(X)$ if and only if $f|_{X_0}$ can be approximated uniformly on $X_0$ by functions in $\hol(X)$.}
\smallskip

To prove our claim, we consider $f|_{X_0}$. Let $g_1(x):=f(x,0)$ and $g_2(x):=f(0,x)$ for $x\in S$. Note that $g_1(0)=g_2(0)=f(0,0)$. Since $g_j\in\cont^0(S)$, there are rational functions $\{R^j_n\}$, $n\in\pnat$, on $\CC$ with no poles on $S$ such that $R^j_n$ converge uniformly to $g_j$ on $S$, $j=1,2$. Now, let
	\bes
		R_n(x,y):=R^1_n(x)+R^2_n(y)-f(0,0)
	\ees
for $(x,y)\in\C^4$ and $n\in\pnat$. Then, $R_n\in \hol(X)$ and, taking uniform limits on $X_0$,
	\beas
		\lim_{n\rightarrow\infty}R_n(x,y)&=&
		\lim_{n\rightarrow\infty}(R^1_n(x)+R^2_n(y)-f(0,0))\\
		&=&
			\begin{cases}
				g_1(x)+g_2(0)-f(0,0),\ \text{when}\ (x,y)\in S\times\{0\};\\
				g_1(0)+g_2(y)-f(0,0),\ \text{when}\ (x,y)\in \{0\}\times S
			\end{cases}
\\
		&=&
			\begin{cases}
				g_1(x),\ \text{when}\ (x,y)\in S\times\{0\};\\
				g_2(y),\ \text{when}\ (x,y)\in \{0\}\times S;\\
			\end{cases}
\\
		&=& f(x,y).
	\eeas
So, $f|_{X_0}$ can be uniformly approximated by elements in $\hol(X)$. By the O'Farrel-Preskenis-Walsh result, $f\in\overline{\hol(X)}$. Since $f\in\cont^0(X)$ was arbitrary, $\cont^0(X)=\overline{\hol(X)}$. Thus, $\cont^0(M)$ has $4$-rational density and $5$-polynomial density.

Now, suppose $M$ admits a totally real immersion into $\C^5$. Then, there is a complex line bundle $Q$ such that $(\C\otimes TM)\oplus Q$ is trivial --- i.e.,
	\bes
		c(\C\otimes TM)\smile c(Q)=1,
	\ees
where $c(B)$ denotes the total Chern class of the vector bundle $B$ and $\smile$ denotes the cup product of cohomology classes. Let $b$ denote the first Chern class of $\C\otimes TS_{_{-3}}$, and $b_1$ and $b_2$ denote the pull-backs of $b$ to $M$ under the corresponding projections to $S_{_{-3}}$. Note that $b$ is nonzero, as $b=w_2 \pmod 2$, where $w_2$ is the second Stiefel-Whitney class of $S_{_{-3}}$ and is the nonzero generator of $H^2(S_{_{-3}};\mathbb{Z}_2)$. Then, $Q$ must satisfy
	\bes
		(1+b_1)(1+b_2)c(Q)=1.
	\ees
But, this means $Q$ is of rank at least $2$, which contradicts the assummption on $Q$. Hence, $M$ does not admit a totally real immersion into $\C^5$. 
\end{proof}

We now paraphrase Corollary ~\ref{cor_opt} for low-dimensional manifolds, and pose some open questions. Note that when $m=1$, the only manifold in question is the circle, any embedding of which is isotropic in $\C$. So, for any $k\geq 0$, the space of $\cont^k$-smooth functions on the circle has $1$-rational density and $2$-polynomial density, and this is optimal.

\medskip

\noindent{\bf Surfaces.} The case of compact surfaces is explored in \cite{ShSu15}, where it is shown that the space of continuous functions on any surface other than $S^2$ and $\RP^2$ has $2$-rational density. For topological reasons no compact surface can have $2$-polynomial density.
Corollary ~\ref{cor_opt} says that the situation for the two exceptions is not too bad since the space of continuous (or $\cont^k$-smooth) functions on any surface has $3$-polynomial density (and this is sharp). We point the reader towards Izzo and Stout's paper \cite[Section~13]{IzSt15} for a possible construction of a PD-basis. The torus $\mathbb T^2$ and all nonorientable compact surfaces with negative Euler characteristic that is a multiple of $4$ admit a smooth Lagrangian embedding into $(\C^2,\om\std)$ (see~\cite{Giv}), and therefore, their $\cont^k$-spaces have $2$-rational density. All other surfaces do not admit a smooth Lagrangian embedding into $\CC$ (see \cite{Au90} or \cite{NeSi16}, the difficult case of the Klein bottle was resolved by Shevchishin~\cite{She}). So, any RD-basis for their $\cont^k$-spaces must contain three elements. These arguments cannot be used for $k=0$, and this raises the following 

\begin{ques}
Is there a rationally convex topological $2$-sphere in $\CC$?
\end{ques}  

\noindent Clearly, such an embedding, if exists, cannot be totally real or have isolated elliptic complex points, and therefore must have either singularities or nongeneric complex points.

\medskip

\noindent{\bf Three-folds.} For $3$-manifolds $M$, Corollary ~\ref{cor_opt} guarantees $4$-polynomial (and rational) density of $\cont^k(M)$, $k\geq 0$. This is optimal for rational density when $k>0$, since $M=\RP^2\times S^1$ does not admit a totally real immersion into $\C^3$ (see \cite{HoJaLa12}). Therefore, $\cont^k(M)$, $k>0$, has $4$-rational density, but not $3$-rational density. For a U-parallelizable (and orientable) example, consider $M=S^3$. If $\cont^k(S^3)$, $k\geq 1$, has an RD-basis of length $3$, then by Theorem ~\ref{thm_main}, $S^3$ admits a $\cont^\infty$-embedding in $\Cn$ that is Lagrangian with respect $\om\std$. Due to a result by Gromov \cite{Gr85}, the embedded sphere supports the boundary of a nonconstant holomorphic disc. By Theorem~2.5 in \cite{DuSi95}, this is not possible since $H^1(S^3,\mathbb{Z})=0$. Thus, $\cont^k(S^3)$ cannot have $3$-rational density when $k>0$. Note that our techniques do not exclude the possibility that $\cont^0(M)$ has $3$-rational density for all three manifolds $M$. One approach to this problem would be to seek Givental-type results (see~\cite{Giv}) on Lagrangian inclusions of three-manifolds into $\C^3$.  

\medskip

\noindent{\bf Four-folds.} By Corollary ~\ref{cor_opt}, $\cont^k(M)$ always admits $6$-polynomial density for any $4$-manifold $M$. As $\C\mathbb{P}^2$ does not admit a totally real embedding into $\C^5$, $\cont^k(\C\mathbb{P}^2)$ ($k>0$) cannot have $5$-polynomial density. It follows from Theorem \ref{thm_main} and \cite[ Lemma~4.1]{HoJaLa12} that $\cont^k(M)$, $k>0$, has $5$-rational (and polynomial) density precisely when the first dual Pontryagin class of $M$ vanishes. To see that $5$ is the optimal length of an RD-basis, consider $M=S^4$. $S^4$ has vanishing first dual Pontryagin class (for orientable manifolds this condition is equivalent to U-parallelizability) but, since its Euler characteristic does not vanish, it does not admit a totally real embedding into $\C^4$ (see \cite{We69}). In particular, $\cont^k(S^4)$ ($k>0$) does not have $4$-rational density.  

\medskip

In view of the difference between the $\cont^0(M)$ and $\cont^k(M)$, when $k>0$, we end this discussion with the following question.

\begin{ques}What are the optimal integers $n$ and $n'$ (in terms of $m$) so that every $\cont^0(M$) has $n$-rational density and $n'$-polynomial density?
\end{ques}

\bibliography{rationaldensity}

\begin{thebibliography}{10}

\bibitem{Au90}
Mich{\`e}le Audin.
\newblock Quelques remarques sur les surfaces lagrangiennes de {G}ivental.
\newblock {\em J. Geom. Phys.}, 7(4):583--598, 1990.

\bibitem{AuLaPo94}
Mich{\`e}le Audin, Fran{\c{c}}ois Lalonde, and Leonid Polterovich.
\newblock Symplectic rigidity: Lagrangian submanifolds.
\newblock In {\em Holomorphic curves in symplectic geometry}, pages 271--321.
  Springer, 1994.

\bibitem{CeEl12}
Kai Cieliebak and Yakov Eliashberg.
\newblock {\em From Stein to Weinstein and back: symplectic geometry of affine
  complex manifolds}, volume~59.
\newblock American Mathematical Society, 2012.

\bibitem{DuSi95}
Julien Duval and Nessim Sibony.
\newblock Polynomial convexity, rational convexity, and currents.
\newblock {\em Duke Math. J.}, 79(2):487--513, 1995.

\bibitem{ElMi02}
Yakov Eliashberg and Nikolai Mishachev.
\newblock {\em Introduction to the h-principle}.
\newblock American Mathematical Society Providence, 2002.

\bibitem{Fo86}
Franc Forstneri{\v c}.
\newblock On totally real embeddings into $\mathbb{C}^n$.
\newblock {\em Expo. Math.}, 4:243--255, 1986.

\bibitem{Fo94}
Franc Forstneri{\v c}.
\newblock Approximation by automorphisms on smooth submanifolds of
  $\mathbb{C}^n$.
\newblock {\em Mathematische Annalen}, 300(1):719--738, 1994.

\bibitem{Fo11}
Franc Forstneri{\v c}.
\newblock {\em Stein manifolds and holomorphic mappings}.
\newblock Springer, 2011.

\bibitem{FoRo93}
Franc Forstneri{\v{c}} and Jean-Pierre Rosay.
\newblock Approximation of biholomorphic mappings by automorphisms of
  $\mathbb{C}^n$.
\newblock {\em Inventiones mathematicae}, 112(1):323--349, 1993.

\bibitem{Giv}
Alexander Givental.
\newblock Lagrangian imbeddings of surfaces and the open {W}hitney umbrella.
\newblock {\em Funktional. Anal. i Prizholen}, 96(3):35--41, 1986.

\bibitem{Gr85}
Mikhael Gromov.
\newblock Pseudo-holomorphic curves in symplectic manifolds.
\newblock {\em Invent. Math.}, 82(2):307--347, 1985.

\bibitem{HoJaLa12}
Pak~Tung Ho, Howard Jacobowitz, and Peter Landweber.
\newblock Optimality for totally real immersions and independent mappings of
  manifolds into $\mathbb{C}^n$.
\newblock {\em New York J. Math.}, 18:463--477, 2012.

\bibitem{IzSt15}
Alexander Izzo and Edgar~Lee Stout.
\newblock Hulls of surfaces.
\newblock {\em arXiv preprint arXiv:1505.03939}, 2015.

\bibitem{WoLo09}
Erik L{\o}w and Erlend~Forn{\ae}ss Wold.
\newblock Polynomial convexity and totally real manifolds.
\newblock {\em Complex Variables and Elliptic Equations}, 54(3-4):265--281,
  2009.

\bibitem{Ne08}
Stefan Nemirovski.
\newblock Finite unions of balls in $\mathbb{C}^n$ are rationally convex.
\newblock {\em Russian Math. Surveys}, 63(2):381--382, 2008.

\bibitem{NeSi16}
Stefan Nemirovski and Kyler Siegel.
\newblock Rationally convex domains and singular {L}agrangian surfaces in
  $\mathbb{C}^2$.
\newblock {\em Invent. Math.}, 203(1):333--358, 2016.

\bibitem{OFPrWa84}
Anthony~G. O'Farrell, John~K. Preskenis, and David Walsh.
\newblock Holomorphic approximation in {L}ipschitz norms.
\newblock {\em Proceedings of the conference on Banach algebras and several
  complex variables (New Haven, Conn., 1983) Contemp. Math.}, 32:187--194,
  1984.

\bibitem{RaSi74}
R.~Michael Range and Yum-Tong Siu.
\newblock $\cont^k$-approximation by holomorphic functions and closed forms on
  $\cont^k$-submanifolds of a complex manifold.
\newblock {\em Math. Ann.}, 210(2):105--122, 1974.

\bibitem{ShSu15}
Rasul Shafikov and Alexandre Sukhov.
\newblock Approximation on real surfaces and rational convexity of {L}agrangian
  inclusions.
\newblock {\em arXiv preprint arXiv:1504.02083}, 2015.

\bibitem{She}
V.~V. Shevchishin.
\newblock Lagrangian embeddings of the klein bottle and the combinatorial
  properties of mapping class groups.
\newblock {\em Izv. Ross. Akad. Nauk Ser. Mat.}, 73(4):153--224, 2009.

\bibitem{St07}
Edgar~Lee Stout.
\newblock {\em Polynomial convexity}, volume 261.
\newblock Springer Science \& Business Media, 2007.

\bibitem{We69}
Raymond~O'Neil Wells~Jr.
\newblock Compact real submanifolds of a complex manifold with nondegenerate
  holomorphic tangent bundles.
\newblock {\em Math. Ann.}, 179(2):123--129, 1969.

\end{thebibliography}
\bibliographystyle{plain}
\end{document}